\documentclass[reqno]{amsart}
\usepackage[T1]{fontenc}
\usepackage{dsfont}
\usepackage{mathrsfs}
\usepackage[colorlinks, citecolor=blue, linkcolor=blue]{hyperref}
\usepackage{xcolor}
\usepackage[a4paper,asymmetric]{geometry}
\usepackage{mathscinet}
\usepackage{latexsym}
\usepackage{amsthm}
\usepackage{amssymb}
\usepackage{amsfonts}
\usepackage{amsmath}
\usepackage{longtable}
\usepackage{graphicx}
\usepackage{multirow}
\usepackage{multicol}
\usepackage{amsfonts, amsmath}
\usepackage{latexsym,bm,amsfonts,amssymb,pifont,mathbbol,bbm}
\usepackage{verbatim}

\newtheorem{theorem}{Theorem}[section]
\newtheorem{thm}[theorem]{Theorem}
\newtheorem{lemma}[theorem]{Lemma}
\newtheorem{lem}[theorem]{Lemma}
\newtheorem{remark}[theorem]{Remark}
\newtheorem{proposition}[theorem]{Proposition}

\newtheorem{corollary}[theorem]{Corollary}
\newtheorem{hyp}[theorem]{HYPOTHESIS}
\theoremstyle{definition}

\newtheorem{defn}[theorem]{Definition}
\newtheorem{example}[theorem]{Example}
 \newtheorem{notation}{Notation}
\theoremstyle{remark}
\numberwithin{equation}{section}

 \DeclareMathAlphabet{\mathpzc}{OT1}{pzc}{m}{it}
 \DeclareMathAlphabet{\mathsfsl}{OT1}{cmss}{m}{sl}

\newcommand{\dif}{\mathrm{d}}

\newcommand{\Be}{\begin{equation}}
\newcommand{\Ee}{\end{equation}}
\newcommand{\Bs}{\begin{split}}
\newcommand{\Es}{\end{split}}
\newcommand{\Bes}{\begin{equation*}}
\newcommand{\Ees}{\end{equation*}}
\newcommand{\BT}{\begin{thm}}
\newcommand{\ET}{\end{thm}}
\newcommand{\Bp}{\begin{proof}}
\newcommand{\Ep}{\end{proof}}
\newcommand{\BL}{\begin{lem}}
\newcommand{\EL}{\end{lem}}
\newcommand{\BP}{\begin{proposition}}
\newcommand{\EP}{\end{proposition}}
\newcommand{\BC}{\begin{corollary}}
\newcommand{\EC}{\end{corollary}}
\newcommand{\BR}{\begin{remark}}
\newcommand{\ER}{\end{remark}}
\newcommand{\BD}{\begin{defn}}
\newcommand{\ED}{\end{defn}}
\newcommand{\BI}{\begin{itemize}}
\newcommand{\EI}{\end{itemize}}

\allowdisplaybreaks

\begin{document}
\title[Parameter estimation of non-ergodic Ornstein-Uhlenbeck]{Parameter estimation of non-ergodic Ornstein-Uhlenbeck processes driven by general Gaussian processes}
\author[YP. Lu]{Yanping LU}
 \address{}
\email{luyp@jxnu.edu.cn}
\begin{abstract}

In this paper, we consider the statistical inference of the drift parameter $\theta$ of non-ergodic Ornstein-Uhlenbeck~(O-U) process driven by a general Gaussian process $(G_t)_{t\ge 0}$. When $H \in (0, \frac 12) \cup (\frac 12,1) $ the second order mixed partial derivative of $R (t, s) = E [G_t G_s] $ can be decomposed into two parts, one of which coincides with that of fractional Brownian motion (fBm), and the other of which is bounded by $|ts|^{H-1}$. This condition covers a large number of common Gaussian processes such as fBm, sub-fractional Brownian motion and bi-fractional Brownian motion. Under this condition, we verify that $(G_t)_{t\ge 0}$ satisfies the four assumptions in references \cite{El2016}, that is, noise has H\"{o}lder continuous path; the variance of noise is bounded by the power function; the asymptotic variance of the solution $X_T$ in the case of ergodic O-U process $X$ exists and strictly positive as $T \to \infty$; for fixed $s \in [0,T)$, the noise $G_s$ is asymptotically independent of the ergodic solution $X_T$ as $T \to \infty$, thus ensure the strong consistency and the asymptotic distribution of the estimator $\tilde{\theta}_T$ based on continuous observations of $X$. Verify that $(G_t)_{t\ge 0}$ satisfies the assumption in references \cite{Es-Sebaiy2019}, that is, the variance of the increment process $\{ \zeta_{t_i}-\zeta_{t_{i -1}}, i =1,..., n \}$ is bounded by the product of a power function and a negative exponential function, which ensure that $\hat{\theta}_n$ and $\check{\theta}_n $ are strong consistent and the sequences $\sqrt{T_n} (\hat {\theta}_n - \theta)$ and $\sqrt {T_n} (\check {\theta}_n - \theta)$ are tight based on discrete observations of $X$.\par


\end{abstract}
\maketitle

\section{ Introduction}\label{sec 03}

In this paper, we consider the statistical inference of the nonergodic O-U process $X=\{X_t,t \in[0, t]\}$ defined by the following stochastic differential equation:
\begin{equation}\label{OUmedel}
\mathrm{d} X_t=\theta X_t\mathrm{d} t+ \mathrm{d} G_t, \quad X_0 = 0,
\end{equation}Where $G= (G_t)_{t \ge 0}$ is a one-dimensional zero-mean Gaussian process, and $\theta >0$ is an unknown parameter.

First, we review the statistical inference of the drift coefficient $\theta $ under the condition that ($\theta <0$) is traversed by the model ~\eqref{OUmedel}. When $G$ is a fractional Gaussian process, citation \cite{hu2010, Sottinen2018, hunua2019, Caiwang2022, Kleptsyna2002} and its references study the strong consistency and asymptotic normality of least-squares estimators (LSE), moment estimators and maximum likelihood estimators based on continuous time observations of $X$ respectively.
\cite{XIAO2011, Hu2013} and their references study the asymptotic behavior of LSE based on discrete-time observations of $X$. \cite{Chenzhou2021} and \cite{chengu2021} have studied $H\in(\frac 12, 1)$ and $H\in(0, \frac 12)$ strong consistency, asymptotic normal Berry-ess\'een bound of $X$ continuous time observation for LSE sum and moment estimation, Where \cite{chengu2021} gives the Berry-ess\'een bound which requires $H \in (0,\frac 38)$.

For the case where model ~\eqref{OUmedel} is not ergodic ($\theta >0$). Based on continuous time observations of $X$,  \cite{Belfadli2011} and \cite{Mendy2013} respectively study that when $G$ is a fBm and sub-fractional Browne motion with ~Hurst parameter $H\in (\frac 12,1)$, the strong consistency and asymptotic distribution of the LSE of the drift coefficient $\theta$, \cite{El2016} generalize the results to the general form, and give sufficient conditions based on the properties of $G$. Three examples of fBm with general Hurst parameter, subfractional Brownian motion and double fractional Brownian motion are given. There are also many studies based on discrete time observations of $X$. For example, \cite{Essebaiy2014} and \cite{Cheng2017} respectively study the weighted fractional Browne motion when $G$ is $H\in (\frac 12,1)$ and the parameters are $-1 <a<0,-a<b<a+1$. The strong consistency and random sequence compactness of two LSEs of the drift coefficient $\theta$, then the results in \cite{Es-Sebaiy2019} are generalized to the general form, and sufficient conditions based on the properties of $G$ are given. Three examples of fBm with general ~Hurst parameter, subfractional Brownian motion and double fractional Brownian motion are given. \cite{Alsenafi2021} is studied based on the continuous and discrete time $X $ observation, and \cite{Cheng2017} will the result in promoting to the parameters for $a > 1, | b| < 1 $ and ~ $|b |  < a + 1 $.

For a given noise $G$, prove the strong consistency and asymptotic distribution of the LSE of the drift coefficient $\theta$ based on the continuous time observation of $X$, and prove the strong consistency and random sequence compactness of the two LSEs based on the discrete time observation of $X$. It is necessary to test the four hypotheses in \cite{El2016} and the third hypothesis in \cite{Es-Sebaiy2019}. If the noise $G$ is a Gaussian process listed below, it needs to be tested separately. This work is undoubtedly tedious and adds a lot of unnecessary calculations to the study of drift coefficient estimation for non-ergodic O-U processes. Therefore, this paper presents a more concise method, that is, including a large number of Hypothesis \ref{hypthe1} of Gaussian processes.

\begin{hyp}\label{hypthe1}
For $H \in (0,\frac 12) \cup (\frac 12 ,1)$ , the covariance function $R(s,t)=E (G_tG_s)$ satisfies that
\begin{enumerate}
  \item[(1)] for any $s \ge 0$ , $R(s,0)=0$ .
  \item [(2)]for any fixed $s \in(0,T)$ , $R(t,s)$ is a continuous function on $[0,T]$ which is differentiable with respect to $t (0,s)\cup (s,T)$ , such that $\frac {\partial}{\partial t} R(s,t)$ is absolutely integrable.
  \item[(3)] for any fixed $t \in(0,T)$ , the difference
  $$\frac {\partial R(s,t)}{\partial t} - \frac {\partial R_{H}(s,t)}{\partial t}$$  is a continuous function on $[0,T]$ which is differentiable with respect to $s$ in $(0,T)$ such that $\Phi (s,t)$, the partial derivative with respect to $s$ of the difference, satisfies
  \begin{equation}\label{phi}
  |\Phi(s,t)| \le C'_H  |ts |^{H-1},
  \end{equation}
  where the constant $C'_H \ge 0$ do not dependent on $T$ , and $R_H(s,t)$ is the covariance function of the fBm.
\end{enumerate}
\end{hyp}

Examples of fBm, subfractional Brownian motion, and other Gaussian processes that satisfy the Hypothesis ~\ref{hypthe1} are given below, as well as some that do not.

Based on continuous and discrete time observations of $X$, the least square technique is used to construct the estimators of the drift coefficient $\theta$. Firstly, reference \cite{El2016} studies the ~LSE of the drift coefficient $\theta$ based on continuous time observations of $X$, which is defined as follows
\begin{equation}\label{hattheta}
\tilde{\theta}_T=\frac{\int_0^T X_s \mathrm{d} X_s}{\int_0^T X_s^2 \mathrm{d}  s}=\frac{X_T^2 }{2\int_0^T X_s^2 \mathrm{d} s}.
\end{equation}
Based on the discrete time observation of ~$X$, it is assumed that the process given by the model ~\eqref{OUmedel} ~$X$is observed at equal distance in time with step size $\Delta_n: t_i = I \Delta_n, I =0,\cdots,n$, and $T_n = n\Delta_n$ represents the length of "observation window", where when $n \to \infty$, $\Delta_n \to 0$ and $n\Delta_n \to \infty$. Then, based on the sampled data $X_{t_i}, I =0,\cdots,n$, consider the following two LSEs
\begin{align}\label{hattheta2}
\hat{\theta}_{n} = \frac{\sum\limits_{i=1}^{n} X_{t_{i-1}}(X_{t_i}- X_{t_{i-1}})}{\Delta_n \sum\limits_{i=1}^n X_{t_{i-1}^2}},
\end{align}and
\begin{align}\label{hattheta3}
\check{\theta}_n = \frac{X_{T_n}^2}{2\Delta_n \sum\limits_{i=1}^n X_{t_{i-1}^2}}.
\end{align}

It is worth noting that $X_t$ in model ~\eqref{OUmedel} can be expressed as follows
\begin{align}\label{solution1}
X_t=e^{\theta t} \int_{0}^t e^{-\theta s} \mathrm{d} G_s, \quad t \in[0,T].
\end{align}The integral about $G$ is interpreted as Young meaning.
Suppose the lemma ~\ref{strong1} holds. By using (a-1) of the reference \cite{El2016}, we can obtain
\begin{align}\label{solution2}
X_t =G_t +\theta e^{\theta t}Z_t, \quad t \in[0,T],
\end{align}where
\begin{align}\label{solution21}
Z_t:=\int_{0}^t e^{-\theta s} G_s \mathrm{d} s, \quad t \in[0,T].
\end{align}In addition, the following process is given:
\begin{align}\label{solution12}
\zeta_t := \int_{0}^t e^{-\theta s} \mathrm{d} G_s, \quad t \in[0,T].
\end{align}

The strong consistency and asymptotic distribution of the estimator $\tilde{\theta}_T$ for continuous time observation based on $X$ are given below, and for discrete time observation based on $X$, estimators  $\hat{\theta}_n $ and  $\check{\theta}_n$ the strong consistency and the random sequence of $\sqrt {T_n} (\hat {\theta} _n - \theta) $ and $\sqrt {T_n}(\check{\theta}_n - \theta)$ is tight.

\begin{thm}\label{strongthm}
When Hypothesis~\ref{hypthe1} is satisfied, the estimator $\tilde{\theta}_T$ is given by the formula
~\eqref{hattheta}. Then, when $T \to \infty$,
\begin{align}\label{strong11}
\tilde{\theta}_T \longrightarrow \theta \qquad a.s..
\end{align}
In addition, when $T \to \infty$
\begin{align}\label{asymptotic1}
e^{\theta t}\left( \tilde{\theta}_T -\theta \right) \stackrel{ {law}}{\longrightarrow} \frac{2\sigma_G}{\sqrt{E(Z_{\infty}^2)}} \mathcal{C}(1),
\end{align}where $\sigma_G =\sqrt{\frac{H\Gamma(2H)}{\theta^{2H}}}$ , integral ~$Z_{\infty}=\int_{0}^{\infty} e^{-\theta s} G_s \dif s$ is well-defined (see Lemma 2.1 of \cite{El2016}) , and $\mathcal{C}(1)$ is the standard Cauchy distribution with the probability density function $\frac {1}{\pi(1+x^2)} :x \in R$.
\end{thm}
\begin{thm}\label{strongthm2}
When Hypothesis~\ref{hypthe1} is satisfied, the estimators $\hat{\theta}_n$ and $\check{\theta}_n$ are given by the formulas \eqref{hattheta2} and \eqref{hattheta3}. If when $n \to \infty$, $\Delta_n \to 0$ and for some $\alpha >0$ , $n\Delta_n^{1+\alpha} \to \infty$ , then,
\begin{align}\label{strong12}
\hat{\theta}_n \longrightarrow \theta , \quad \check{\theta}_n \longrightarrow \theta  \qquad a.s. .
\end{align}And, for any $q \ge 0$ ,
\begin{align}\label{asymptotic2}
\Delta_{n}^q e^{\theta T_n}(\hat{\theta}_n - \theta) ~~\text{and}~~ \Delta_{n}^q e^{\theta T_n}(\check{\theta}_n - \theta) ~\text{is not tight}.
\end{align}
In addition, when $n \to \infty$, $n\Delta_{n}^3 \to 0$ , in a sense, the estimators $\hat{\theta}_n$ and $\check{\theta}_n$ is $\sqrt{T_n}\text{consistent}$ , the sequence
\begin{align}\label{asymptotic21}
\sqrt{T_n}(\hat{\theta}_n - \theta) ~\text{and}~ \sqrt{T_n}(\check{\theta}_n - \theta)~ \text{is tight}.
\end{align}
\end{thm}
Next, we give some procedures that satisfy the Hypothesis ~\ref{hypthe1}.
\begin{example}
The fBm $\{B_t^H, t\geq 0\}$ with parameter $H \in (0,1)$ has the covariance function
\begin{equation}\label{fbm}
R_{H}(s,t)= \frac{1}{2}(s^{2H} + t^{2H} - |t-s|^{2H}).
\end{equation}
\end{example}

\begin{example}
\cite{Tomasz2004} The sub-fractional Brownian motion $\{S_t^H, t \geq 0\}$ with parameter $H \in (0,1)$ has the covariance function
$$R(s,t)=s^{2H}+t^{2H}-\frac{1}{2}\left((s+t)^{2H}+|t-s|^{2H}\right).$$
\end{example}

\begin{example}
\cite{Lei2009,Bardina2011} The bi-fractional Brownian motion $\{B_t^{H,K}, t\geq 0\}$ with parameter $H \in (0, 1), K \in (0, 2)$ and $HK\in (0,1)$ has the covariance function
$$R(s,t)=\frac{1}{2^K}\left((s^{2H}+t^{2H})^K - |t-s|^{2HK}\right) .$$
\end{example}

\begin{example}
\cite{Sgha2013} The generalized sub-fractional Brownian motion $\{S_t^{H,K},t \ge 0 \} $ with parameter $H \in (0, 1), K \in (0, 2)$ and $HK\in (0,1)$ has the covariance function
$$ R(s, t)= (s^{2H}+t^{2H})^{K}-\frac12 \big[(t+s)^{2HK} + |t-s|^{2HK} \big].$$
\end{example}

\begin{example}
\cite{Bojdecki2007} The gaussian process $\{G_t , t\ge 0\} $ with parameter $H \in (0, 1)$ has the covariance function
$$ R(s, t)= (t+s)^{2H} - |t-s|^{2H} .$$
\end{example}

\begin{example}
\cite{Talarczyk2020} The gaussian process $\{G_t , t\ge 0\} $ with parameter $\gamma \in (0, 1)$ has the covariance function
$$ R(s, t)= (\max(s,t))^{\gamma} - |t-s|^{\gamma} .$$
\end{example}

\begin{example}
\cite{Houdre2003} The gaussian process $\{G_t , t\ge 0\} $ with parameter $H \in (0, 1),\,K \in(0,1)$ has the covariance function
$$ R(s, t)= \frac 12 \left[t^{H}+s^{H}-K(t+s)^{H} -(1-K) |t-s|^{ H}\right] .$$
\end{example}

Below are some procedures that do not satisfy the Hypothesis ~\ref{hypthe1}.
\begin{example}
\cite{Bojdecki2007} The weighted fractional Brown motion $\{B_t^{a,b}, t \geq 0\}$ with parameter $a>-1,|b|<1$ and $|b|<a+1$ has the covariance function
$$R(s,t)=\int_{0}^{s\wedge t}u^a \left[ (t-u)^b +(s-u)^b  \right] \mathrm{d} u.$$
It is known from \cite{Alsenafi2021} that can be reduced to
$$R(s,t)=\beta(a+1,b+1)[t^{a+b+1}+s^{a+b+1}]-\int_{s\wedge t}^{s\vee t}u^a (t \vee s -u)^b \mathrm{d} u,$$ where $\beta(\cdot,\cdot)$ denotes the beta function. $(s\wedge t)^{-a}$ of its second mixed partial derivative coincides with fBm, so it does not satisfy Hypothesis \ref{hypthe1}; \cite{Alsenafi2021} gives a hypothesis to prove the statistical inference of LSE based on continuous and discrete time observations of $X$, which is different from the hypothesis in \cite{El2016} and \cite{Es-Sebaiy2019}.
\end{example}

\begin{example}
\cite{Lei2009,Bardina2009} The gaussian process $\{G_t , t\ge 0\} $ with parameter $H \in (0, 1),\,K \in(0,2)$ has the covariance function
$$X_{t}^{K} = \int_{0}^{\infty} \left( 1-e^{-r t} \right) r^{-\frac {1+K}{2}} \mathrm{d} W_{r} ,\quad t\ge 0 ,$$  where $\{W_t,t \ge 0 \}$ is the standard Brownian motion. The covariance function is
\begin{equation*}
R_{K}(s,t)=\mathrm{Cov} (X_{s}^{K},X_{t}^{K})=\left\{
      \begin{array}{ll}
\frac {\Gamma(1-K)}{K}[t^{K}+s^{K} -(t+s)^{K}], & \quad  K\in (0,1) ;\\
\frac {\Gamma(2-K)}{K(K-1)}[(t+s)^{K}-t^{K}-s^{K} ], &\quad K\in (1,2).
 \end{array}
\right.
\end{equation*}The second order mixed partial derivative is only to $| ts | ^ {H - 1} $i s bounded part, so don't say to meet \ref {hypthe1}; and when $T \to \infty$, the asymptotic variance of the solution $X_T$ in the case of the O-U process $X$ traversal is 0, which does not satisfy the $\sigma_G >0$ of \cite{El2016}{$(\mathcal{H}3)$}. However, it satisfies three hypotheses in \cite{Es-Sebaiy2019}.
\end{example}

\begin{example}
\cite{Talarczyk2020} The gaussian process $\{G_t , t\ge 0\} $ with parameter $\gamma \in (0, 1)$ has the covariance function
$$ R(s, t)= (t+s)^{\gamma}-(\max(s,t))^{\gamma}.$$
The second order mixed partial derivative is only to $| ts | ^ {H - 1} $ is bounded part, so don't say to meet Hypotheses \ref{hypthe1}; and when $T \to \infty$, the asymptotic variance of the solution $X_T$ in the case of the O-U process $X$ traversal is 0, which does not satisfy the $\sigma_G >0$ of \cite{El2016}{$(\mathcal{H}3)$}. However, it satisfies three hypotheses in \cite{Es-Sebaiy2019}.
\end{example}

\begin{remark}
Through the examples listed above, found that based on $X $ discrete observations, even if the covariance function of second order mixed partial derivative is only to $| ts | ^ {H - 1} $is bounded part, also meet \cite{Es-Sebaiy2019} in the three hypotheses. However, based on the continuous observation of $X$, we do not find any Gaussian process satisfying the four assumptions in \cite{El2016} but not the Hypothesis ~\ref{hypthe1}. In addition, if the model \eqref{OUmedel} is driven by a linear combination of independent zero-mean Gaussian processes, each of which satisfies the Hypothesis \ref{hypthe1}, The estimator $\tilde{\theta}_T$ based on continuous time observation $X$ has strong consistency and asymptotic distribution. And based on the discrete time $X $ observation estimator $\hat{\theta}_n$ and $\check{\theta}_n$ has strong consistency and random sequence $\sqrt{T_n}(\hat{\theta}_n - \theta)$ and $\sqrt{T_n}(\check{\theta}_n - \theta)$is tight.
\end{remark}

\section{Preliminary}
Denote $G = \{G_t, t\in [0,T] \}$ as a continuous centered Gaussian process with covariance function
$$ E (G_tG_s)=R(s,t), \ s, t \in [0,T], $$
defined on a complete probability space $(\Omega, \mathcal{F}, P)$. The filtration $\mathcal{F}$ is generated by the Gaussian family $G$. Suppose, in addition, that the covariance function $R$ is continuous. Let $\mathcal{E}$ denote the space of all real valued step functions on $[0,T]$. The Hilbert space $\mathfrak{H}$ is defined as the closure of $\mathcal{E}$ with the inner product
\begin{align*}
\langle \mathbf{1}_{[a,b)},\,\mathbf{1}_{[c,d)}\rangle _{\mathfrak{H}}=E \left(( G_b-G_a) ( G_d-G_c) \right).
\end{align*}
Where $\mathbf{1}_{[a,b)}$ represents the indicator function of $[a,b)$. We denote $G=\{G(h), h \in \mathfrak{H}\}$ as the isonormal Gaussian process on the probability space $(\Omega, \mathcal{F}, P)$, indexed by the elements in the Hilbert space $\mathfrak{H}$. In other words, $G$ is a Gaussian family of random variables such that
$$\mathbb{E}(G(h)) = 0, \quad \mathbb{E}(G(g)G(h)) = \langle g, h \rangle_{\mathfrak{H}} .$$
for any $g, h \in \mathfrak{H}$.

\begin{notation}\label{notation1}
Let $R_H(s,t)$ be the covariance function of the fractional Brownian motion as in \eqref{fbm}. $\mathfrak{H}_1$ is the separable Hilbert space associated with fBm. $\mathcal{V}_{[0,T]}$ denote the set of bounded variation function on $[0,T]$. For function $f,g \in \mathcal{V}_{[0,T]}$, we define two products as
\begin{align}\label{neiji11}
\langle f,g \rangle_{\mathfrak{H}_1} = - \int_{[0,T]^2} f(t) \frac{\partial R_H (s,t)}{\partial t} \mathrm{d} t \nu_{g}( \mathrm{d} s),
\end{align}
\begin{align}\label{neiji12}
\langle f,g\rangle_{\mathfrak{H}_2} = C'_{H} \int_{[0,T]^2} |f(t)g(s)| (ts)^{H-1} \mathrm{d} t \mathrm{d} s,
\end{align}where $\nu_{g}$ is given below.
\end{notation}

The following proposition is an extension of Theorem 2.3 of \cite{Jolis2007} and from Proposition 2.1 of \cite{Chending2021}, which gives the products representation of the Hilbert space $\mathfrak{H}$:
\begin{proposition}\label{main prelim}
$\mathcal{V}_{[0,T]}$ denote the set of bounded variation function on $[0,T]$. Then, $\mathcal{V}_{[0,T]}$ is dense in $\mathfrak{H}$ and we have 	\begin{align} \label{innp fg30}
\langle f,g \rangle _{\mathfrak{H}} &=\int_{[0,T]^2} R(s,t) \nu_f( \mathrm{d} t) \nu_{g}( \mathrm{d} s), \quad \forall f,g \in \mathcal{V}_{[0,T]},
\end{align}
where $\nu_{g}$ is the restriction to $([0,T],\mathcal{B}([0,T]))$ of the Lebesgue-Stieljes signed measure associated with $g^0$ define as
\begin{equation*}
g^0(x)=\left\{
      \begin{array}{ll}
 g(x), & \quad \text{if }~ x\in [0,T] ;\\
0, &\quad \text{otherwise }.
 \end{array}
\right.
\end{equation*}In addition, if the covariance function $R(s,t)$ satisfies the Hypothesis ~\ref{hypthe1}, then
\begin{align}\label{neiji3}
\langle f,g \rangle _{\mathfrak{H}} = - \int_{[0,T]^2} f(t) \frac{\partial R(s,t)}{\partial t} \mathrm{d} t \nu_{g}( \mathrm{d} s),
\end{align}and
\begin{align}\label{neiji4}
|\langle f,g \rangle_{\mathfrak{H}}-\langle f,g \rangle_{\mathfrak{H}_1} | \le \langle f,g \rangle_{\mathfrak{H}_2}.
\end{align}
\end{proposition}

\begin{corollary}\label{tuilunneiji}
Denote by $\delta_a(\cdot)$ the Dirac $\delta$ function centered at a point $a$. Let $f=h_1 \cdot \mathbf{1}_{[a,b)}(\cdot)$, $g=h_2 \cdot \mathbf{1}_{[c,d)}(\cdot)$ , with $h_1$ and $h_2$ are continuously differentiable function. Then we have
\begin{align}\label{neiji}
\langle f,g \rangle_{\mathfrak{H}} &=-\int_{[0,T]^2} h_1(t)\mathbf{1}_{[a,b)}(t) h_2'(s) \mathbf{1}_{[c,d)}(s) \frac{\partial R(s,t)}{\partial t} \mathrm{d} t \mathrm{d} s \nonumber  \\
&+ \int_{[0,T]^2} h_1(t)\mathbf{1}_{[a,b)}(t) h_2(s) \frac{\partial R(s,t)}{\partial t} [\delta_b(s) - \delta_a(s)] \mathrm{d} t \mathrm{d} s.
\end{align}
\end{corollary}

\begin{remark}
The inequality \eqref{neiji4} is the starting point of the present paper, which. when $H\in (\frac 12 ,1)$, Hypothesis \ref{hypthe1} (3) imply that the identity \eqref{neiji3} can be rewritten as
\begin{align}\label{neiji5}
\langle f,g \rangle_{\mathfrak{H}} =\int_{[0,T]^2} f(t)g(s) \frac{\partial^2 R(s,t)}{\partial t \partial s} \mathrm{d} t \mathrm{d} s \quad \forall f,g \in \mathcal{V}_{[0,T]}.
\end{align}In the case, inequality \eqref{neiji4} has been obtained in \cite{Chenzhou2021}. when $H \in (0,\frac 12)$, it is well known that both $\frac{\partial^2 R(u,v)}{\partial u \partial v}$  and $\frac{\partial^2 R_H (u,v)}{\partial u \partial v}$ are not absolutely integrable. But the absolute integrability of their difference makes the key inequality \eqref{neiji4} still valid.
\end{remark}

However, if $f$ and $g$ support on the disjoint interval, the formula \eqref{neiji5} also holds in $H \in(0, frac 12)$ when the process $g$ is fBm. The lemma given below is for Lemma 2.1 of \cite{Cheridito2003} extension, proof method is the same as formula \eqref{neiji4}.
\begin{lemma}\label{aassum3}
Let $f,g$ be bounded variances supported on $[a,b]$ and $[c,d]$ respectively. when $H\in (0,\frac 12)\cup (\frac 12,1)$ and $0 \le a<b \le c<d<\infty$, we have
$$\langle f,g \rangle_{\mathfrak{H}}=\int_{c}^d \int_{a}^b f(t)g(s)\frac{\partial^2 R(s,t)}{\partial t \partial s} \mathrm{d} t \mathrm{d} s.$$
In addition, if the covariance function $R(s,t)$ satisfies the Hypothesis \ref{hypthe1}, then
$$|\langle f,g \rangle_{\mathfrak{H}}|\le \int_{c}^d \int_{a}^b |f(t)g(s)|\left[C_H(s-t)^{2H-2}+C'_H(ts)^{H-1} \right] \mathrm{d} t \mathrm{d} s,$$ where $C_H=|H(2H-1)|$ .
\end{lemma}

The following theorem is a description of the main results in the references \cite{El2016} and \cite{Es-Sebaiy2019}.
\begin{thm}\label{ELtheorem}
In model \eqref{OUmedel}, the following conditions are given for process $G$:
\begin{enumerate}
  \item[(1)] The process $G$ has  H\"{o}lder continuous paths of order $\delta \in (0,1]$.
  \item [(2)]for every $t \ge 0$ , $E(G_t^2) \le c t^{2\gamma}$ for some positive constants $c$ and $\gamma$.
  \item [(3)]The limiting variance of $e^{-\theta T} \int_{0}^T e^{\theta s} \dif G_s$ exist and strictly positive as $T \to \infty$, i.e., there exists a constant $\sigma_G >0$ such that
     \begin{align*}
      \lim_{T \to \infty} E \left[ \left(e^{-\theta T} \int_{0}^T e^{\theta s} \dif G_s \right)^2 \right] = \sigma_G^2.
     \end{align*}
  \item [(4)] For all fixed $s\in[0,T)$,
     \begin{align*}
      \lim_{T \to \infty} E \left( G_s e^{-\theta T} \int_{0}^T e^{\theta r }\dif G_r \right) =0.
     \end{align*}
  \item [(5)] For each $i=1,\cdots,n$ and $n \ge 1$ , existence of normal number $\gamma$, have
$$E \left[ (\zeta_{t_i}-\zeta_{t_{i-1}})^2 \right]\le C(\gamma)\Delta_n^{2\gamma} e^{-2\theta t_i} . $$
\end{enumerate}If the process $G$ satisfy conditions $(1)$ and $(2)$, the strong consistency \eqref{strong11} holds; If the process $G$ satisfy conditions $(1)-(4)$, the asymptotic distribution \eqref{asymptotic1} holds; If the process $G$ satisfy conditions $(1)$, $(2)$ and $(5)$, then the strong consistency \eqref{strong12} and the random sequence compactness \eqref{asymptotic2},
\eqref{asymptotic21} holds.
\end{thm}

The following inequality comes from Lemma 3.3 of \cite{Chenzhou2021}.
\begin{lemma}\label{chengu}
Assuming $\beta \in (0, 1) $, there is a constant $C>0$, for any $t \ge 0$, we have
\begin{align}\label{chengushizi}
\int_{0}^t e^{-\theta(t-r)} r^{\beta-1}\mathrm{d} r \le C(1 \wedge t^{\beta-1}).
\end{align}
\end{lemma}
For the rest of this article, $C$ will be a normal number that does not depend on $T$, and its value may vary from line to line.

\section{Proof of Theorem \ref{strongthm} and Theorem \ref{strongthm2}}
First, define some functions that will be used in the proof. Let
\begin{align*}
& h(\cdot) = \mathbf{1}_{[s,T)} (\cdot), \quad f(\cdot)=e^{-\theta(T-\cdot)} \mathbf{1}_{[0,T)}(\cdot), \\
& g(\cdot)= \mathbf{1}_{[0,s)} (\cdot), \quad m(\cdot)=e^{-\theta(\cdot)} \mathbf{1}_{[t_{i-1},t_i)}(\cdot).
\end{align*}

The following Lemmas \ref{strong1}, \ref{assum1} and \ref{assum2} show that the assumptions in \cite{El2016} are valid if \ref{hypthe1} is valid. Lemmas \ref{strong1} and \ref{assum3} indicates that if \ref{hypthe1} holds, the assumptions in \cite{Es-Sebaiy2019} holds.

\begin{lemma}\label{strong1}
Under Hypothesis \ref{hypthe1}, there exists a constant $C>0$ that does not depend on $T$, for all $s, t \ge 0$, there is
\begin{align}\label{strong2}
E [|G_t-G_s|^2] \le C|t-s|^{2H}.
\end{align}Specially, when $s=0$, $E(G_t^2) \le Ct^{2H}$ .
\end{lemma}
\begin{proof}
Suppose $t \ge s$. According to It\^{o} isometric and inequality \eqref{neiji4}, there is
$$E [|G_t-G_s|^2] =\langle h,h \rangle_{\mathfrak{H}} \le \langle h,h \rangle_{\mathfrak{H}_1}+\langle h,h \rangle_{\mathfrak{H}_2} .$$
where $\langle h,h \rangle_{\mathfrak{H}_1}$ is shown in equation \eqref{neiji11}. Because $R_H (t, s) $ is covariance function of fBm, so $\langle h,h \rangle_{\mathfrak{H}_1}$ value is $|t-s|^{2H}$. The formula \eqref{neiji12} can be obtained
$$\langle h,h \rangle_{\mathfrak{H}_2} = C'_H \left(\int_{s}^t u^{H-1} \dif u\right)^2 =C(t^H - s^H)^2 \le  C |t-s|^{2H},$$
We're using an inequality $1-x^H \le (1-x)^H$ here, where $x \in [0,1]$ and $H \in (0,1)$. Therefore, we have
$$E [|G_t-G_s|^2] \le C|t-s|^{2H},$$and when $s=0$, we have $E(G_t^2) \le Ct^{2H}$ .
\end{proof}

\begin{lemma}\label{assum1}
The constant $\sigma_G =\sqrt{\frac{H\Gamma(2H)}{\theta^{2H}}}$. Under Hypothesis \ref{hypthe1}, we have that
\begin{align}\label{assum11}
\lim_{T \to \infty} E \left[ \left(e^{-\theta T} \int_{0}^T e^{\theta s} \dif G_s \right)^2 \right] = \sigma_G^2.
\end{align}
\end{lemma}
\begin{proof}
According to It\^{o} isometric, there is
$$E \left[ \left(e^{-\theta T} \int_{0}^T e^{\theta s} \dif G_s \right)^2 \right] = \langle f,f \rangle_{\mathfrak{H}} .$$

When $H \in (0,\frac 12)\cup(\frac 12,1)$, Lemma 2.3 of \cite{chenli2021} gives the limit when $T \to \infty$:
$$\langle f,f \rangle_{\mathfrak{H}_1} \longrightarrow  \sigma_G^2.$$
In addition, by the formula \eqref{neiji12} and Lemma \ref{chengu}, we have
\begin{align*}
\langle f,f \rangle_{\mathfrak{H}_2}&= C'_H \int_{[0,T]^2} e^{-\theta(T-u)}e^{-\theta(T-v)} (uv)^{H-1} \mathrm{d} u \mathrm{d} v \\
&= C'_H \left( e^{-\theta T} \int_{0}^{T} e^{\theta u}u^{H-1}\dif u \right)^2  \le CT^{2H-2}.
\end{align*}
Therefore, according to the inequality \eqref{neiji4} and the triangle inequality, when $T \to \infty$, we have
\begin{align*}
|\langle f,f \rangle_{\mathfrak{H}}-\sigma_G^2| & \le |\langle f,f \rangle_{\mathfrak{H}}- \langle f,f \rangle_{\mathfrak{H}_1}|+|\langle f,f \rangle_{\mathfrak{H}_1}- \sigma_G^2 |  \\
& \le \langle f,f \rangle_{\mathfrak{H}_2} + |\langle f,f \rangle_{\mathfrak{H}_1}- \sigma_G^2 | \\
 & \longrightarrow  0 .
\end{align*}
\end{proof}

\begin{lemma}\label{assum2}
Under Hypothesis \ref{hypthe1}. For any fixed $ s\in [0,T)$ , we have
\begin{align}\label{assum21}
\lim_{T \to \infty} E \left( G_s e^{-\theta T} \int_{0}^T e^{\theta r }\dif G_r \right) =0.
\end{align}
\end{lemma}
\begin{proof}
For fixed $s\in [0,T)$, divide the function $f(\cdot)$ into the following form
\begin{align*}
f(\cdot) = e^{-\theta(T-\cdot)}\mathbf{1}_{[0,s)}(\cdot)+e^{-\theta (T-\cdot)}\mathbf{1}_{[s,T)}(\cdot) :=f_{1}(\cdot)+f_{2}(\cdot).
\end{align*}
According to It\^{o} isometric, there is
\begin{align}\label{assum32}
E \left( G_s e^{-\theta T} \int_{0}^T e^{\theta r }\dif G_r \right) = \langle f,g \rangle_{\mathfrak{H}} =\langle f_1,g\rangle_{\mathfrak{H}}+\langle f_2,g\rangle_{\mathfrak{H}}.
\end{align}
First calculate the limit of $\langle f_1,g\rangle_{\mathfrak{H}}$. According to formulas \eqref{neiji} and \eqref{neiji12}, monotonicity indicate the presence of normal number $C_{\theta,H,s}$ and $C'_{\theta,H,s}$, have
\begin{align*}
0<\langle f_1,g \rangle_{\mathfrak{H}_1} &= e^{-\theta T} \int_{R^2}e^{\theta v} \mathbf{1}_{[0,s)}(v) (\delta_{s}(u)-\delta_{0}(u)) \frac {\partial R_H (u,v)}{\partial v} \mathrm{d} u \mathrm{d} v \\
& =H e^{-\theta T} \int_{0}^{s}e^{\theta v}(v^{2H-1} +(s-v)^{2H-1}) \mathrm{d} v \\
& \le H e^{-\theta T} e^{\theta s}\int_{0}^{s} (v^{2H-1} +(s-v)^{2H-1}) \mathrm{d} v \\
& =C_{\theta,H,s}  e^{-\theta T},
\end{align*}and
\begin{align*}
0<\langle f_1,g \rangle_{\mathfrak{H}_2} &= C'_H e^{-\theta T} \int_{0}^s \int_{0}^s e^{\theta v} (uv)^{H-1}\mathrm{d} u \mathrm{d} v \\
&\le C'_H e^{-\theta T} e^{\theta s} \left(\int_{0}^s u^{H-1}\mathrm{d} u \right)^2 \\
&= C'_{\theta,H,s} e^{-\theta T}.
\end{align*}Therefore, when $T \to \infty$, inner products $\langle f_1,g \rangle_{\mathfrak{H}_1}$ and $\langle f_1,g \rangle_{\mathfrak{H}_2}$ converge to 0. By substituting it into \eqref{neiji4}, it can be obtained by forced convergence
\begin{align*}
\langle f_1,g \rangle_{\mathfrak{H}}  \longrightarrow 0 \quad (T \to \infty).
\end{align*}

Next, calculate the limit of $\langle f_2,g\rangle_{\mathfrak{H}}$. Because of the function $f_2$ and $g$ were support on $[s,T)$ and $[0,s)$, and $H \in (0,\frac 12)\cup(\frac 12 ,1)$, so according to Lemma \ref{aassum3}, we have
$$|\langle f_2,g \rangle_{\mathfrak{H}}| \le  \int_{s}^{T} \int_{0}^s e^{-\theta(T-v)}\left[ C_H(v-u)^{2H-2}+C'_H(uv)^{H-1} \right] \mathrm{d} u \mathrm{d} v. $$
Using the L'H\^{o}spital rule
\begin{align*}
\lim\limits_{T \to \infty} |\langle f_2,g \rangle_{\mathfrak{H}}|&\le \lim\limits_{T \to \infty}  e^{-\theta T} \int_{s}^{T} \int_{0}^s e^{\theta v}\left[ C_H(v-u)^{2H-2}+C'_H(uv)^{H-1} \right] \mathrm{d} u \mathrm{d} v \\
& = \lim\limits_{T \to \infty}\theta^{-1}  \int_{0}^s \left[ C_H(T-u)^{2H-2}+C'_H (Tu)^{H-1} \right] \mathrm{d} u\\
&= 0.
\end{align*}

Finally, combining the limits of $\langle f_1,g\rangle_{\mathfrak{H}}$ and $\langle f_2,g\rangle_{\mathfrak{H}}$ with formula \eqref{assum32}, the desired result \eqref{assum21} is obtained.
\end{proof}

\noindent{\it proof of Theorem \ref{strongthm}\,}
According to the Kolmogorov continuity criterion and the Lemma \ref{strong1}, we know that for all $\varepsilon \in (0,H)$, the process $G$ has $H-\varepsilon $ order H\"{o}lder continuous path, and the strong consistency of the estimator $\tilde{\theta}_T$ is deduced from the Theorem \ref{ELtheorem}. In addition, combining the results of Lemma \ref{strong1}, Lemma \ref{assum1} and Lemma \ref{assum2}, the Theorem \ref{ELtheorem} can derive its asymptotic distribution.

\begin{lemma}\label{lemmaassum3}
Assuming $\gamma \in (0,1)$ , $t_i = i \Delta_n$ , and $\Delta_n \to 0$ . Then, for any $i=1,...,n$ , there is
\begin{align}\label{lemmaassum31}
e^{t_i} \int_{t_{i-1}}^{t_i} e^{-\theta u} (t_i -u)^{2\gamma -1} \dif u \le C(\gamma) \Delta_n^{2\gamma} e^{-2\theta t_{i-1}},
\end{align}
\begin{align}\label{lemmaassum32}
e^{t_{i-1}} \int_{t_{i-1}}^{t_i} e^{-\theta u} (u-t_{i-1})^{2\gamma -1} \dif u \le C(\gamma) \Delta_n^{2\gamma} e^{-2\theta t_{i-1}}.
\end{align}
\end{lemma}
\begin{proof}
The proof process is similar to Lemma 4.2 of \cite{Es-Sebaiy2019}, which will not be proved in detail here.
\end{proof}

\begin{lemma}\label{assum3}
Under Hypothesis \ref{hypthe1}, for each $i=1,\cdots,n$ and $n \ge 1$ , we have that
\begin{align}\label{assum31}
E\left[ (\zeta_{t_i}-\zeta_{t_{i-1}})^2 \right]\le C(H)\Delta_n^{2H}  e^{-2\theta t_i}.
\end{align}
\end{lemma}
\begin{proof}
According to It\^{o} isometric and inequality \eqref{neiji4}, there is
\begin{align}\label{assum33}
E \left[ (\zeta_{t_i}-\zeta_{t_{i-1}})^2 \right]=\langle m,m \rangle_{\mathfrak{H}}\le \langle m,m \rangle_{\mathfrak{H}_1}+\langle m,m \rangle_{\mathfrak{H}_2}.
\end{align}Again, first calculate the value of ~$\langle m,m \rangle_{\mathfrak{H}_1}$. When the functions $f,g $ in the Corollary \ref{tuilunneiji} are equal, there is
\begin{align*}
\langle m,m \rangle_{\mathfrak{H}_1}
&=  \int_{[t_{i-1},t_i]^2} e^{-\theta (u+v)} \frac{\partial R_H(u,v)}{\partial u}\big[\theta+( \delta_{t_i}(v) -\delta_{t_{i-1}}(v)) \big] \mathrm{d} u \mathrm{d} v \\
& = \theta H \int_{t_{i-1}}^{t_i}\int_{t_{i-1}}^{t_i} e^{-\theta (u+v)} \big( u^{2H-1} -|u-v|^{2H-1}\mathrm{sgn}(u-v) \big) \mathrm{d} u \mathrm{d} v \\
&+ H  e^{-\theta t_i} \int_{t_{i-1}}^{t_i}  e^{-\theta u} \big( u^{2H-1}+(t_i -u)^{2H-1} \big) \mathrm{d} u \\
&- H e^{-\theta t_{i-1}} \int_{t_{i-1}}^{t_i}  e^{-\theta u} \big( u^{2H-1}-(u-t_{i-1})^{2H-1} \big) \mathrm{d} u \\
& = H  e^{-\theta t_i} \int_{t_{i-1}}^{t_i}  e^{-\theta u} (t_i -u)^{2H-1}  \mathrm{d} u + H e^{-\theta t_{i-1}} \int_{t_{i-1}}^{t_i}  e^{-\theta u} (u-t_{i-1})^{2H-1}  \mathrm{d} u \\
&\le  e^{-2\theta t_{i-1}} \Delta_n^{2H}.
\end{align*}The Lemma \ref{lemmaassum3} holds when $H=\gamma$. Next, calculate the value of $\langle m,m \rangle_{\mathfrak{H}_2}$. By
the formula \eqref{neiji12}, we have
\begin{align*}
\langle m,m \rangle_{\mathfrak{H}_2}&= C'_H \int_{t_{i-1}}^{t_i}\int_{t_{i-1}}^{t_i}  e^{-\theta (u+v)} (uv)^{H-1} \mathrm{d} u \mathrm{d} v \\
&= C'_H \left(\int_{t_{i-1}}^{t_i}   e^{-\theta u} u^{H-1}  \mathrm{d} u \right)^2 \\
& \le \frac{C'_H}{H^2}   e^{-2\theta t_{i-1}} \Delta_n^{2H}.
\end{align*} The desired result is obtained by substituting the above two inequalities into \eqref{assum33}.
\end{proof}

\noindent{\it proof of the Theorem \ref{strongthm2}\,}
Combining the results of the Lemma \ref{strong1} and the Lemma \ref{assum3}, the strong consistency of $\hat{\theta}_n$ and $\check{\theta}_n$ is deduced by \ref{ELtheorem}, which is based on $X$ discrete time observation. And random sequence $\sqrt{T_n}(\hat{\theta}_n - \theta)$ and $\sqrt{T_n}(\check{\theta}_n - \theta)$ is tight.


\begin{thebibliography}{10}
\bibitem{El2016} El Machkouri M, Es-Sebaiy K, Ouknine Y. Least squares estimator for non-ergodic Ornstein-Uhlenbeck processes driven by Gaussian processes[J]. Journal of the Korean Statistical Society, 2016, 45(3): 329-341.
\bibitem{Es-Sebaiy2019} Es-sebaiy K, Alazemi F,  Al-foraih M. Least Squares Type Estimation for Discretely Observed Non-Ergodic Gaussian Ornstein-Uhlenbeck Processes[J]. Acta Math Sci, 2019, 39(4): 989-1002.
\bibitem{hu2010} Hu Y, Nualart D. Parameter estimation for fractional Ornstein-Uhlenbeck processes[J]. Stat Probab Lett, 2010, 80(11-12): 1030-1038.
\bibitem{hunua2019} Hu Y, Nualart D, Zhou H. Parameter estimation for fractional Ornstein-Uhlenbeck processes of general Hurst parameter[J]. Stat Inference Stoch Process, 2019, 22(1): 111-142.
\bibitem{Kleptsyna2002} Kleptsyna M L, Lebreton A. Statistical analysis of the fractional Ornstein-Uhlenbeck type process[J]. Journal of Stat Inference Stoch Process, 2002, 5: 229-248.
\bibitem{Sottinen2018} Sottinen T, Viitasaari L. Parameter estimation for the Langevin equation with stationary-increment Gaussian noise[J]. Stat Inference Stoch Process, 2018, 21(3): 569-601.
\bibitem{Caiwang2022} Cai C, Wang Q, Xiao W. Mixed Sub-fractional Brownian Motion and Drift Estimation of Related Ornstein-Uhlenbeck Process[J]. Communications in Mathematics and Statistics, 2022: 1-27.
\bibitem{XIAO2011}Xiao W L, Zhang W G, Xu W D. PParameter estimation for fractional Ornstein-Uhlenbeck processes at discrete observation[J]. Math. Model, 2011, 35(9): 4196-4207.
\bibitem{Hu2013} Hu Y Z, Song J. Parameter estimation for fractional Ornstein-Uhlenbeck processes with discrete observations[J]. Springer Proc. Math. Stat, 2013, 34: 427-442.
\bibitem{Chenzhou2021} Chen Y, Zhou  H.  Parameter estimation for an ornstein-uhlenbeck process driven by a general gaussian noise[J]. Acta Mathematica Scientia, 2021, 41(2): 573-595.
\bibitem{chengu2021} Chen Y, Gu X, Li Y. Parameter estimation for an Ornstein-Uhlenbeck Process driven by a general Gaussian noise with Hurst Parameter $ H\in (0,\frac12) $[J]. arXiv preprint arXiv:2111.15292, 2021.
\bibitem{Belfadli2011} Belfadli R, Es-sebaiy K, Ouknine Y.Parameter estimation for fractional Ornstein-Uhlenbeck Processes: Non-ergodic case[J]. Frontiers in Science and Engineering (An International Journal Edited by Hassan II Academy of Science and Technology), 2011, 1(2): 1-16.
\bibitem{Mendy2013} Mendy I. Parametric estimation for sub-fractional Ornstein-Uhlenbeck process[J]. Journal of Statistical Planning and Inference, 2013, 143(4): 663-674.
\bibitem{Essebaiy2014} Es-sebaiy K, Ndiaye D. On drift estimation for discretely observed non-ergodic fractional Ornstein Uhlenbeck processes with discrete observations[J]. Afr Stat, 2014, 9(1):615-625.
\bibitem{Cheng2017} Cheng P, Shen G, Chen Q. Parameter estimation for nonergodic Ornstein-Uhlenbeck process driven by the weighted fractional Brownian motion[J]. Adv. Difference Equ, 2017, 2017(1): 366.
\bibitem{Alsenafi2021} Alsenafu A, Al-foraih M, Es-sebaiy K. Least squares estimation for non-ergodic weighted fractional Ornstein-Uhlenbeck process of general parameters[J]. AIMS Math, 2021,  6(11): 12780-12794.
\bibitem{Tomasz2004} Tomasz B, Luis G, Anna T. Sub-fractional Brownian motion and its relation to occupation times[J]. Statist. Probab. Lett, 2004, 69(4): 405-419.
\bibitem{Lei2009} Lei P, Nualart D. A decomposition of the bi-fractional Brownian motion and some applications[J]. Statist. Probab. Lett, 2009, 79(5): 619-624.
\bibitem{Bardina2011} Bardina X, Es-sebaiy K. An extension of bifractional Brownian motion[J]. Commun. Stoch. Anal, 2011,  5(2): 333-340.
\bibitem{Sgha2013} Sghir A. The generalized sub-fractional Brownian motion[J]. Communications on Stochastic Analysis, 2013,  7(3): 2.
\bibitem{Bojdecki2007} Bojdecki T,  Gorostiza L,  Talarczyk A. Some extensions of fractional Brownian motion and sub-fractional Brownian motion related to particle systems[J].   Electronic communications in probability,  2007,  12: 161-172.
\bibitem{Talarczyk2020} Talarczyk A. Bifractional Brownian motion for $H<1$ and $2HK \le 1$[J].   Statistics and Probability Letters,  2020,  157: 108628.
\bibitem{Houdre2003}  Houdr\'e C, Villa J. An example of infinite dimensional quasi-helix[J].     Contemporary Mathematics,  2003,  336: 195-202.
\bibitem{Bardina2009} Bardina X, Bascompie D. A decomposition and weak approximation of the sub-fractional Brownian motion[J]. Preprint, Departament de Matematiques, UAB,  2009.
\bibitem{Jolis2007} Jolis M. On the Wiener integral with respect to the fractional Brownian motion on an interval[J]. J  Math  Anal  Appl, 2007, 330(2): 1115-1127.
\bibitem{Chending2021} Chen Y, Ding Z, Li Y. Berry-Esseen bounds and almost sure CLT for the quadratic variation of a general Gaussian noise[J]. arXiv: 2106.01851.
\bibitem{Cheridito2003}Cheridito P, Kawaguchi H, Maejima M. Fractional Ornstein-Uhlenbeck processes[J].  Electron. J. Probab, 2003,  8(3):205-211.
\bibitem{chenli2021} Chen Y, Li Y. Berry-esseen bound for the parameter estimation of fractional ornstein-uhlenbeck processes with $0 < H < \frac 12$[J].  Communications in Statistics-Theory and Methods, 2021,  50(13): 2996-3013.
\end{thebibliography}
\end{document}